\documentclass[12pt, reqno]{amsart}
\usepackage{amssymb,amscd,amsthm, verbatim,amsmath,color,fancyhdr, mathrsfs}
\usepackage{graphicx}
\usepackage{turnstile,cite}
\usepackage[plainpages=false,pdfpagelabels]{hyperref}
\usepackage{hyperref}
\usepackage{orcidlink}
\usepackage{setspace}
\usepackage{enumitem}
\allowdisplaybreaks

\hypersetup{colorlinks,%
	citecolor=red,%
	filecolor=black,%
	linkcolor=blue,%
	urlcolor=blue
}
\usepackage[varg]{pxfonts}

\allowdisplaybreaks
\usepackage[letterpaper, left=2.5cm, right=2.5cm, top=2.5cm,
bottom=2.5cm,dvips]{geometry}

\newtheorem{theorem}{Theorem}[section]
\newtheorem{corollary}[theorem]{Corollary}
\newtheorem{lemma}[theorem]{Lemma}
\newtheorem{proposition}[theorem]{Proposition}
\theoremstyle{definition}
\newtheorem{definition}[theorem]{Definition}
\newtheorem{remark}[theorem]{Remark}
\newtheorem{example}[theorem]{Example}
\numberwithin{equation}{section}
\begin{document}
	\title{Fixed Point Theorems for Paired Contractive Mappings in $b$-Metric Spaces}
	
	\author{Hezekiah Seun Adewinbi ${^{1*}}$\orcidlink{0009-0003-0643-2146}}
	\address{${^1}$Department of Mathematical Sciences, Kent State University,
		Kent, OH, 44242, United States.}
	\email{hadewinb@kent.edu}
	
	\keywords{Paired-Kannan, Paired-Chatterjea, $b$-metric space, Paired contraction, fixed point \\
		{\rm 2020} {\it Mathematics Subject Classification}: 54H25, 47H10\\
		*Corresponding author: Hezekiah Seun Adewinbi (hadewinb@kent.edu)}

\begin{abstract}
In this paper, we develop a fixed point framework for paired
contractive mappings in complete $b$-metric spaces. We introduce the
class of paired-Chatterjea contractions and establish a
corresponding fixed point theorem. We further show that every
Chatterjea-type contraction is a paired-Chatterjea contraction,
whereas the converse implication does not hold, demonstrating that
the paired formulation properly extends the Chatterjea-type class.
Additionally, we establish a fixed point theorem for Paired contractions under the condition $\alpha\in[0,1)$, thereby extending the admissible range of the contraction constant from $[0,1/s)$ in Bouker et al.~(2026) to $[0,1)$. Furthermore, we examine the Paired–Kannan fixed point result of Bouker et al.~(2026) and provide an example showing that the stated coefficient condition does not guarantee the claimed fixed point conclusion in $b$-metric spaces. Motivated by this observation, we establish a corrected Paired–Kannan fixed point theorem under a complete and appropriately formulated set of hypotheses. Under suitable contractive conditions, we prove that the contractions under consideration have fixed points if and only if they have no periodic point of prime period $2$.  Moreover, we show that each of these contractions has at most two fixed points. Several examples are presented to illustrate our results and distinguish the proposed paired contractive mappings from existing ones.
\end{abstract}

\maketitle
\section{Introduction}
The celebrated Banach contraction principle \cite{Banach} is one of the cornerstones of nonlinear analysis and fixed point theory. It guarantees the existence and uniqueness of fixed points for contraction mappings defined on complete metric spaces and provides an effective iterative procedure for approximating such fixed points. Owing to its elegance and wide applicability, the Banach contraction principle has inspired numerous extensions and generalizations \cite{Patel, Ju, Eroglu, Fatai2026, 	Proinov2020, Popescu2021, Karapinar2017, Kikina, Jleli2026, Jleli2025}. 

In an effort to broaden the scope of the classical metric framework, Bakhtin\cite{Bakhtin} and  Czerwik \cite{Czerwik} introduced the notion of a $b$-metric space.  This generalization significantly enlarges the class of spaces in which fixed point techniques can be applied. Since its introduction, the theory of $b$-metric spaces has attracted considerable attention, leading to numerous extensions of classical fixed point results and the development of new contractive conditions tailored to this generalized setting \cite{Guan2024, Derouiche2020, Bota2020, Mitrovic2025, Chaib, Chaib2026, Aleksic2018, Huang2018, Shatnawi}.

In 2024, Chand and Rohen \cite{Chand2024} introduced the concept
of paired contractions in complete metric spaces, providing a new
framework for studying self-mappings satisfying a joint contractive condition involving three points. Their work opens new avenues for extending classical fixed point theory beyond traditional single-map contractions (see \cite{Chand2025, Chiroma2026,Chand2024b, Banerjee2026}).

Motivated by these developments, this paper investigates paired contractive mappings in $b$-metric spaces. First, we introduce Paired–Chatterjea contraction in $b$-metric spaces and establish corresponding fixed point results. Second, we establish fixed point results for Paired contraction in $b$-metric spaces by extending the admissible range of the contraction constant from $[0,1/s)$, as assumed in \cite[Theorem 2]{Bouker2026}, to the broader interval $[0,1)$. Third, we examine the coefficient restriction in the Paired–Kannan contraction and provide an explicit example showing that the condition stated in \cite[Theorem 3]{Bouker2026} is insufficient to guarantee the claimed fixed point conclusion in $b$-metric spaces. We subsequently establish the corresponding fixed point result under the revised condition
$\alpha<\min\left\{\frac12,\frac1s\right\}$. We include several examples to illustrate the results and distinguish the proposed paired contractive mappings from existing classes.

\section{Preliminaries}
To make the paper self-contained and facilitate the reader's understanding of what follows, we first recall several key definitions and foundational results. 
	\begin{definition} (See \cite{Bakhtin} and \cite{Czerwik})
	\label{b-metric space}
	Let $X$ be a nonempty set, and let $s\geq 1$ be a given real number. A function $d:X\times X\to\mathbb{R}_+$ is called a \textit{$b$-metric} if for all $x,y,z\in X$,
	\begin{enumerate}[label=(\alph*)]
		\item $d(x,y)=0$ if and only if $x=y$;
		\item $d(x,y)=d(y,x)$;
		\item $d(x,z)\leq s[d(x,y)+d(y,z)]$.
	\end{enumerate}
	A pair $(X,d)$ is called a $b$-metric space. 
\end{definition}

\begin{example}
	Let $(X, d_0)$ be a metric space, and let $n > 1$. Define  
	$d(x, y) = d_0(x, y)^n.$ Thus, $(X, d)$ is a $b$-metric space with constant $s = 2^{n- 1}$.
\end{example}

\begin{definition}(See \cite{Bakhtin} and \cite{Czerwik})
	Let $(X,d)$ be a $b$-metric space and $\{x_n\}$ be a sequence in $X$ and $x \in X$. Then
	\begin{enumerate}[label=(\alph*)]
		\item $\{x_n\}$ is called convergent to a point $x$ if and only if
		$\displaystyle	\lim_{n \to \infty} d(x_n, x) = 0.$
		\item $\{x_n\}$ is a Cauchy sequence if and only if
		$ \displaystyle d(x_n, x_m) \to 0,~\text{as } n, m \to \infty.$
		\item  $(X,d)$ is said to be complete if and only if every Cauchy sequence in $X$ is convergent.
	\end{enumerate}
\end{definition}

\begin{definition}
	Let $X$ be a nonempty set and let $T:X\to X$ be a self-mapping on $X$. A point $x^*\in X$ is called a \textit{fixed point} of $T$ if
	$Tx^*=x^*.$ The set of all fixed points of $T$ is denoted by
	$	\operatorname{Fix}(T)=\{x\in X:Tx=x\}.$
\end{definition}

\begin{theorem}(See \cite[Theorem 2.1]{Chand2024})
	Let $(X,d)$ be a complete metric space with a cardinality $|X| \geq 3$. A mapping $T : X \to X$ is said to have the $PC$ property on $X$ if there exists a real number $\alpha \in [0,1)$ such that
	\begin{equation}
			\label{Equation Paired-contraction}
		d(Tx, Ty) + d(Ty, Tz) \leq \alpha(d(x, y) + d(y, z))
	\end{equation}
	for all  pairwise distinct $x, y, z \in X$. If $T$ has the $PC$ property on $X$, then the following hold:
	\begin{enumerate}
		\item[(i)] $T$ has a fixed point if and only if $T$ does not have periodic points of prime period 2.
		\item[(ii)] The number of fixed points is at most two.
	\end{enumerate}
\end{theorem}

\begin{definition}(See \cite[Definition 4]{Bouker2026})
	\label{def4}
	Let $(X, d)$ be a $b$-metric space with a cardinality $|X| \geq 3$. The self-mapping $T : X \to X$ is referred to as a \textit{Paired Contraction} in $b$-metric space if there is a constant $\alpha \in [0, \frac{1}{s})$ so that the inequality \eqref{Equation Paired-contraction}
	holds for all pairwise distinct $x, y, z \in X$.
\end{definition}

\begin{theorem} (See \cite[Theorem 2]{Bouker2026})
	Let $(X, d)$ be a complete $b$-metric space with a cardinality $|X| \geq 3$, and consider a mapping $T : X \to X$ satisfy the following two conditions
	\begin{enumerate}
		\item[(a)] $T$ is a Paired Contraction mapping in $b$-metric space.
		\item[(b)] There are no periodic points of prime period 2 for $T$.
	\end{enumerate}
	Then $T$ has a fixed point. The number of fixed points is not more than two.
\end{theorem}

\begin{definition}(See \cite[Definition 5]{Bouker2026})
	\label{Withouts}
	Let $(X, d)$ be a $b$-metric space with a cardinality $|X| \geq 3$. The self-mapping $T : X \to X$ is referred to as a \textit{Paired-Kannan contraction mapping} in $b$-metric space if there is a constant $\alpha \in [0, \frac{1}{2})$, so that the inequality
	\begin{equation}
			\label{Equation paired-Kannan}
		d(Tx, Ty) + d(Ty, Tz) \leq \alpha \left(d(x, Tx) + d(y, Ty) + d(z, Tz)\right),
	\end{equation}
	holds for all pairwise distinct $x, y, z \in X$.
\end{definition}

\begin{theorem}(See \cite[Theorem 3]{Bouker2026})
	Let $(X, d)$ be a complete $b$-metric space with a cardinality $|X| \geq 3$, and consider a mapping $T : X \to X$ satisfy the following two conditions
	\begin{enumerate}
		\item[(a)] $T$ is a Paired-Kannan contraction mapping in $b$-metric space.
		\item[(b)] There are no periodic points of prime period 2 for $T$.
	\end{enumerate}
	Then $T$ has a fixed point. The number of fixed points is not more than two.
\end{theorem}

\begin{theorem}(See \cite[Theorem 2.2]{Mitrovic})
	Let $(X, d)$ be a complete $b$-metric space and $T : X \to X$ be a mapping satisfying 
\begin{equation}
	\label{Kannan}
		d(Tx,Ty)\leq\alpha(d(x,Tx)+d(y,Ty))
\end{equation}
	for all $x,y \in X$.  Then $T$ has a unique fixed point if $\alpha< \min\left\{\frac{1}{2},\frac{1}{s}\right\}$. 
\end{theorem}

\begin{theorem}(See \cite[Theorem 3]{Kir2013})
	\label{Theorem1.1}
	Let $(X, d)$ be a complete $b$-metric space with constant $s \geq 1$. Suppose $T: X \to X$ is a mapping satisfying the \textit{Chatterjea-type contraction} condition
	\begin{equation}
		\label{Chatterjea}
		d(Tx, Ty) \leq \alpha \left[ d(x, Ty) + d(y, Tx) \right]
	\end{equation}
	for all $x, y \in X$, where $\alpha \in \left[0, \frac{1}{2s}\right)$. Then $T$ has a unique fixed point in $X$.
\end{theorem}

\begin{lemma}(See \cite[Lemma 2.2]{Miculescu})
	\label{lem:cauchy}
	Let $(X,d)$ be a $b$-metric space with constant $s \geq 1$. Suppose that $\{x_n\}$ is a sequence in $X$. If there exists $\gamma \in [0,1)$ satisfying
	$$
	d(x_{n+1}, x_n) \leq \gamma d(x_n, x_{n-1}),
	$$
	for every $n \in \mathbb{N}$, then $\{x_n\}$ is a Cauchy sequence.
\end{lemma}

\section{paired-Chatterjea Contractions}
In this section, we introduce the Paired–Chatterjea contraction in $b$-metric spaces. We show that every Chatterjea-type contraction is a Paired-Chatterjea contraction and provide an example demonstrating that the converse implication does not hold. Finally, we establish several fixed point theorems for Paired–Chatterjea contractions in $b$-metric spaces.

\begin{definition}
	\label{PairedChar}
	Let $(X, d)$ be a $b$-metric space with $|X| \geq 3$. A mapping $T: X \to X$ is called a \textit{paired-Chatterjea contraction} on $X$ if there exists a constant 
	$\alpha\in\left[0,\frac{1}{2s}\right)$
	such that 
	\begin{eqnarray}
		\label{Equation paired-Chatterjea}
		d(Tx, Ty) + d(Ty, Tz) \leq \alpha(d(x, Ty) + d(y, Tx) + d(y, Tz) + d(z, Ty))
	\end{eqnarray}
	holds for all three pairwise distinct $x,y,z\in X.$
\end{definition}

\begin{remark}
	Let $(X, d)$ be a $b$-metric space with $|X| \geq 3$, $T: X \to X$ be a Chatterjea-type contraction on $X$ and let $x, y, z \in X$ be pairwise distinct. Consider the inequality \eqref{Chatterjea} for the pair $y, z$:
	\begin{align}
		\label{Chatterjea1}
		d(Ty, Tz) &\leq \alpha (d(y, Tz) + d(z, Ty))
	\end{align}
	Adding the left and the right parts of the inequalities \eqref{Chatterjea} and \eqref{Chatterjea1}, we obtain \eqref{Equation paired-Chatterjea}. Hence, every Chatterjea-type contraction is a Paired-Chatterjea contraction.
\end{remark}

\begin{lemma}
	\label{Cauchy lemma}
	Let $(X,d)$ be a $b$-metric space with $|X|\geq 3$ and let
	$T:X\to X$ be a Paired-Chatterjea contraction. Suppose that $\{x_n\}$ is a
	sequence of pairwise distinct points in $X$ satisfying
	$x_{n+1}=Tx_n,~ n\in \mathbb{N}_0=\mathbb{N}\cup \{0\}$.
	Then $\{x_n\}$ is a Cauchy sequence.
\end{lemma}

\begin{proof}
	Since $x_{n-1},x_n,x_{n+1}$ are pairwise distinct, setting $x=x_{n-1}$, $y=x_n$, and $z=x_{n+1}$ in \eqref{Equation paired-Chatterjea} and applying the $b$-metric inequality, we obtain
	\begin{align*}
		d(x_n,x_{n+1})+d(x_{n+1},x_{n+2})&\leq\alpha (d(x_{n-1},x_{n+1})+d(x_n,x_{n+2}))\\
		&\leq s\alpha (d(x_{n-1},x_{n})+d(x_n,x_{n+1})+d(x_n,x_{n+1})+d(x_{n+1},x_{n+2})).
	\end{align*}
	Thus, 
	\begin{align}
		d(x_n,x_{n+1})+d(x_{n+1},x_{n+2})&\leq\frac{\alpha s}{1-\alpha s}(d(x_{n-1},x_n)+d(x_n,x_{n+1}))\nonumber\\
		\label{same}
		d(x_n,x_{n+1})+d(x_{n+1},x_{n+2})&\leq\lambda(d(x_{n-1},x_n)+d(x_n,x_{n+1})),
	\end{align}
	where $\lambda=\frac{\alpha s}{1-\alpha s}<1$ since $\alpha<\frac{1}{2s}$.  	Let 
	$	\delta_n=d(x_n,x_{n+1}),~ n\in \mathbb{N}_0.	$
	Thus, the inequality \eqref{same} gives
	$$
	\delta_n+\delta_{n+1}
	\leq
	\lambda(\delta_{n-1}+\delta_n),
	~n\geq1.
	$$
	It follows inductively that
	$
	\delta_n+\delta_{n+1}
	\leq
	\lambda^n(\delta_0+\delta_1),
	~n\in \mathbb{N}_0.
	$
	Consequently,
	\begin{eqnarray}
		\lim_{n\to\infty}\delta_n=0.
		\label{C1}
	\end{eqnarray}
	
	We now prove that $\{x_n\}$ is Cauchy. Suppose, to the contrary,
	that it is not Cauchy. Then there exist $\varepsilon>0$ and sequences
	of integers $\{m_k\}$ and $\{n_k\}$ such that
	$m_k>n_k,~n_k\to\infty,$
	and
	\begin{eqnarray}
		d(x_{m_k},x_{n_k})\geq\varepsilon
		\label{C2}
	\end{eqnarray}
	for every $k$. In view of \eqref{C1}, we may assume, after passing to
	a subsequence if necessary, that $m_k\geq n_k+2.$
	Since the orbit points are pairwise distinct, the three points
	$x_{m_k-1},~ x_{n_k-1},~ x_{n_k}$
	are pairwise distinct. Hence, by \eqref{Equation paired-Chatterjea}, we have
	\begin{align}
		d(x_{m_k},x_{n_k})+d(x_{n_k},x_{n_k+1})&\leq
		\alpha\left(
		d(x_{m_k-1},x_{n_k})
		+d(x_{n_k-1},x_{m_k})
		+d(x_{n_k-1},x_{n_k+1})
		+d(x_{n_k},x_{n_k})
		\right)\nonumber\\
		\label{C4}
		d(x_{m_k},x_{n_k})+\delta_{n_k}
		&\leq\alpha\left(
		d(x_{m_k-1},x_{n_k})
		+d(x_{n_k-1},x_{m_k})
		+d(x_{n_k-1},x_{n_k+1})
		\right).
	\end{align}
	By the $b$-metric inequality,
	$$
	d(x_{m_k-1},x_{n_k})
	\leq
	s\left(
	d(x_{m_k-1},x_{m_k})
	+d(x_{m_k},x_{n_k})
	\right)
	\text{ and }
	d(x_{n_k-1},x_{m_k})
	\leq
	s\left(
	d(x_{n_k-1},x_{n_k})
	+d(x_{n_k},x_{m_k})
	\right).
	$$
	Moreover,
	$$
	d(x_{n_k-1},x_{n_k+1})
	\leq
	s\left(
	d(x_{n_k-1},x_{n_k})
	+d(x_{n_k},x_{n_k+1})
	\right).
	$$
	Therefore, \eqref{C4} becomes
	\begin{align*}
		d(x_{m_k},x_{n_k})+\delta_{n_k}
		\leq
		\alpha s\Bigl[
		2d(x_{m_k},x_{n_k})
		+\delta_{m_k-1}
		+2\delta_{n_k-1}
		+\delta_{n_k}
		\Bigr].
	\end{align*}
	Consequently,
	\begin{align}
		\label{C5}
		(1-2s\alpha)d(x_{m_k},x_{n_k})
		\leq
		s\alpha\left(
		\delta_{m_k-1}
		+2\delta_{n_k-1}
		+\delta_{n_k}
		\right).
	\end{align}
	Since $m_k,n_k\to\infty$, \eqref{C1} implies
	$
	\delta_{m_k-1}\to0,~
	\delta_{n_k-1}\to0,~
	\delta_{n_k}\to0.
	$
	Because
	$
	1-2s\alpha>0,
	$
	it follows from \eqref{C5} that
	$$
	\lim_{k\to\infty}d(x_{m_k},x_{n_k})=0.
	$$
	This contradicts \eqref{C2}. Hence $\{x_n\}$ is a Cauchy
	sequence.
\end{proof}

\begin{theorem}
	\label{Theorem paired-Chatterjea}
	Let $(X,d)$ be a complete $b$-metric space with $|X| \geq 3$ and let $T: X \to X$ be a Paired-Chatterjea contraction. Then $T$ has a fixed point in $X$ if and only if $T$ has no periodic points of prime period 2. Moreover, the number of fixed points of $T$ is at most two. 
\end{theorem}

\begin{proof}
	Let $T$ be a mapping with no periodic point of prime period 2. Let $x_0\in X$ such that
	$x_1=Tx_0 , ~ x_2=Tx_1,~\ldots, ~ x_n=Tx_{n-1}=T^nx_0 .$
	 If $x_n$ is a fixed point of $T$ for some $n\in \mathbb{N}_0$, then there is nothing to prove. Suppose that $x_n$ is not a fixed point of the mapping $T$ for every $n\in\mathbb{N}_0$, we have $x_{n}\neq Tx_{n}=x_{n+1}$. Additionally, since $T$ has no periodic points of prime period 2, then  $x_{n+1} = T(T(x_{n-1})) \neq x_{n-1}$. Hence $x_{n-1}, x_n, x_{n+1}$ are pairwise distinct.  Putting $x=x_{n-1}, y=x_n, z= x_{n+1}$ in \eqref{Equation paired-Chatterjea}, we get \eqref{same}. Let
	$$p_n = d(x_n, x_{n+1}) + d(x_{n+1}, x_{n+2})$$
	for all $n \in \mathbb{N}$. 
	Since $x_n, x_{n+1}$ and $x_{n+2}$ are pairwise distinct, it follows from \eqref{same} that
	\begin{eqnarray}
		\label{estimate}
		p_1 \leq  \lambda p_0, ~ p_2 \leq  \lambda p_1, ~ p_n \leq  \lambda p_{n-1} \text{ and }
	p_0 > p_1 > \dots > p_n > \dots.
\end{eqnarray}

Next, we show that all points in the orbit $\{x_n:n\in\mathbb{N}_0\}$ are
distinct. Suppose, to the contrary, that $x_i=x_j$ for some
$0\leq i<j$. If $j=i+1$, then $x_i=Tx_i$, contradicting the
assumption that the orbit contains no fixed point. Hence $j\geq i+2$,
and consequently
$x_{j+1}=x_{i+1},~x_{j+2}=x_{i+2}.$
Thus,
\begin{align*}
	p_i
	=d(x_i,x_{i+1})+d(x_{i+1},x_{i+2})
	=d(x_j,x_{j+1})+d(x_{j+1},x_{j+2})
	=p_j,
\end{align*}
which contradicts $p_j<p_i$ for $i<j$ in \eqref{estimate}. Hence, the orbit points are
pairwise distinct. By Lemma \ref{Cauchy lemma}, $\{x_n\}$ is a Cauchy sequence. The completeness of $(X,d)$ implies $x_n\to x^*\in X.$ 
	
	Next, we show that $x^*$ is a fixed point of $T$.  Since the orbit points are pairwise distinct, the point $x^*$ can coincide with at most one point of the orbit. Hence there exists $N \in \mathbb{N}$ such that $x_n \neq x^*, x_{n+1} \neq x^*$
	for all $n \geq N$. Thus, for every $n \geq N$, the three points $x^*, x_n, x_{n+1}$ are pairwise distinct. 
	By \eqref{Equation paired-Chatterjea} and $b$-metric inequality, we obtain
	\begin{align*}
		d(x_{n+1},Tx^*)&=d(Tx_n,Tx^*)\leq d(Tx_n,Tx^*)+d(Tx^*,Tx_{n+1})\\
		&\leq \alpha (d(x_n,Tx^*)+d(x^*,Tx_n)+d(x^*,Tx_{n+1})+d(x_{n+1},Tx^*))\\
		&\leq s\alpha d(x_n,x_{n+1})+s\alpha d(x_{n+1}, Tx^*)+\alpha d(x^*,x_{n+1})+\alpha d(x^*,x_{n+2})+\alpha d(x_{n+1},Tx^*)\\
		&\leq s\alpha d(x_n,x_{n+1})+2s\alpha d(x_{n+1}, Tx^*)+\alpha d(x^*,x_{n+1})+\alpha d(x^*,x_{n+2}).
	\end{align*}
	Thus, we have 
	\begin{eqnarray*}
		(1-2s\alpha)d(x_{n+1}, Tx^*)\leq s\alpha d(x_n,x_{n+1})+\alpha d(x^*,x_{n+1})+\alpha d(x^*,x_{n+2}).
	\end{eqnarray*}
 Therefore, as $n\to \infty$, we obtain
	$\displaystyle	\lim_{n \to \infty}d(x_{n+1}, Tx^*)=0$ since $\alpha< \frac{1}{2s}$.  By the $b$-metric inequality, we have 
	$$d(x^*,Tx^*)\leq s(d(x^*,x_{n+1})+d(x_{n+1},Tx^*)).$$
	Taking the limit as $n \to \infty$, we obtain $d(x^*, Tx^*) = 0$. Therefore, $x^*$ is a fixed point of $T$.
	
	On the other hand, suppose that $T$ has a fixed point $x\in X$. i.e. $Tx=x$ and assume also that $T$ has  a periodic point $y$ of prime period 2. i.e. $T(Ty)=y$. If we let $z=Ty$, then $Tz=y$. If $x=y$ or $y=z$, then $Ty=y$, which contradicts the assumption that $y$ is a periodic point of prime period 2. Also, if $x=z$, then $Ty=z=x=Tx=Tz=y$, which is a contradiction with the same reason as above. Therefore, we may assume that $x,y$ and $z$ are pairwise distinct. In this case, applying inequality \eqref{Equation paired-Chatterjea} and $b$-metric inequality, we obtain
	\begin{align*}
		d(x, z) + d(z, y)& \leq \alpha(d(x, z) + d(y, x))\\
		&\leq \alpha(d(x, z) + sd(y, z)+sd(z,x))\\
		&\leq \alpha (1+s)(d(x, z) + d(y, z)). 
	\end{align*}
	Thus $\alpha\geq \frac{1}{s+1}\geq \frac{1}{2s}$, contradicting the assumption that $\alpha<\frac{1}{2s}$. Hence, $T$ has no periodic point of prime period 2. 
	
	Finally, suppose $T$ has three distinct fixed points, say $Tx = x$, $Ty = y$, and $Tz = z$. Then, \eqref{Equation paired-Chatterjea} yields $\alpha\geq \frac{1}{2}$, contradicting the assumption that $\alpha\in [0,\frac{1}{2s})$. Therefore, $T$ has at most two fixed points.
\end{proof}

The following example shows that the class of paired-Chatterjea contractions properly extends the corresponding class of Chatterjea-type contractions.
\begin{example}
	Let $X = \{a, b, c\}$ and define $d: X \times X \to \mathbb{R}^+$ by
	$$
	d(a, b) = 1,~d(b, c) = 1,~d(a, c) = 6,~d(x,x)=0 \text{ for all } x\in X.
	$$
	Then,  $(X, d)$ is a complete $b$-metric space with $s = 3$. Define $T: X \to X$ by
	$$
	Tx = 
	\begin{cases}
		b, & \text{if } x = a; \\
		c, & \text{if } x \in \{b,c\}; \\
	\end{cases}
	$$
	Clearly $T$  has no periodic points of prime period 2.  By \eqref{Equation paired-Chatterjea}, the six permutations of $(a,b,c)$ yield the following ratios: $(a,b,c)$ and $(c,b,a)$ give $1/7$; $(a,c,b)$ and $(b,c,a)$ give $1/8$; and $(b,a,c)$ and $(c,a,b)$ give $2/13$. Therefore, all the hypotheses of Theorem  \ref{Theorem paired-Chatterjea} are satisfied with any $\alpha\in [\frac{2}{13},\frac{1}{6})$ and $\text{Fix}(T)=\{c\}.$ Next, we show that $T$ is not a Chatterjea-type contraction. Indeed, taking $x=a$ and $y=b$, we obtain
$
d(Ta,Tb)=d(b,c)=1
\text{ and }
d(a,Tb)+d(b,Ta)
=d(a,c)+d(b,b)
=6. 
$
Thus, by \eqref{Chatterjea}, we have
$
\alpha\geq\frac{1}{6},
$
which contradicts the hypothesis of Theorem \ref{Theorem1.1}, namely,
$
\alpha<\frac{1}{2s}=\frac{1}{6}.
$
Therefore, $T$ is not a Chatterjea-type contraction.

\end{example}

\begin{remark} 
Putting $s=1$ in Theorem \ref{Theorem paired-Chatterjea}, we get \cite[Theorem 3]{Chand2025}. 
\end{remark}

\section{Paired Contractions}  
In this section, we introduce Paired contractions for single-valued self-mappings on $b$-metric spaces. We retain the terminology of Definition \ref{def4} while considering a broader class of Paired contractions by allowing $\alpha\in[0,1)$. Furthermore, we establish a fixed point theorem for this class of contractions.

\begin{definition} 
	Let $(X,d)$ be a $b$-metric space with $|X| \geq 3$. A mapping $T: X \to X$ is called \textit{Paired contraction} if there exists a constant $\alpha \in [0, 1)$ such that  \eqref{Equation Paired-contraction} 
	holds for all pairwise distinct $x, y, z \in X$.
\end{definition}

\begin{remark}
The pairwise distinctness of the points $x,y,z\in X$ is essential, since without it, the inequality \eqref{Equation Paired-contraction} reduces to the standard contraction condition.
\end{remark}
\begin{lemma}
	\label{Cauchy-Paired-contraction}
	Let $(X,d)$ be a $b$-metric space with $|X|\geq 3$ and let
	$T:X\to X$ be a Paired contraction with $\alpha\in (0,1)$.  If the orbit $\{x_n:n\geq 0\}$, defined by
	$x_{n+1}=Tx_n$, is pairwise distinct, then $\{x_n\}$ is a Cauchy
	sequence.
\end{lemma}

\begin{proof}
	Let $\delta_n=d(x_n,x_{n+1})$ for all $n\in\mathbb{N}_0.$
	Since the points $x_{n-1},x_n,x_{n+1}$ are pairwise distinct, applying
	\eqref{Equation Paired-contraction} to these three points gives
	\begin{equation}
		\label{Paired-Banach Inequality}
		d(x_n,x_{n+1})+d(x_{n+1},x_{n+2})
	\leq
	\alpha\left(d(x_{n-1},x_n)+d(x_n,x_{n+1})\right).
	\end{equation}
	Hence
	$
		\delta_n+\delta_{n+1}
		\leq
		\alpha(\delta_{n-1}+\delta_n),
		~ n\geq1.$
	Consequently, by induction,
	\begin{equation*}
		\label{A2}
		\delta_n+\delta_{n+1}
		\leq
		\alpha^n(\delta_0+\delta_1),
		~ n\in\mathbb{N}_0.
	\end{equation*}
	In particular,
	\begin{equation}
		\label{A3}
		\delta_n\leq \alpha^n(\delta_0+\delta_1),
		~ n\in\mathbb{N}_0.
	\end{equation}
	Since $\alpha\in(0,1)$,  there exists $n_0\geq 2$ such that
	$s^2\alpha^{n_0}<1.$ We first establish an estimate for shifted orbit points. Let
	$m>n$ and put
	$$
	A_j=d(x_{m+j},x_{n+j}),~ j\geq0.
	$$
	For $j\geq1$ and $m\neq n+1$, the points
	$x_{m+j-1},x_{n+j-1},x_{n+j}$ are pairwise distinct (if $m=n+1$, we have $d(x_m,x_n)=\delta_n\to 0$ as $n\to\infty$). Hence, applying
	\eqref{Equation Paired-contraction} to these points, we obtain
	\begin{align*}
		A_j+\delta_{n+j}
		&\leq
		\alpha\left(A_{j-1}+\delta_{n+j-1}\right).
	\end{align*}
	Therefore,
$A_j	\leq
		\alpha A_{j-1}
		+\alpha\delta_{n+j-1}.$
	Using \eqref{A3}, we obtain
	$
	A_j
	\leq
	\alpha A_{j-1}
	+\alpha^{n+j}(\delta_0+\delta_1).
	$
	Iterating this inequality for $j=1,\ldots,n_0$ yields
	$
	A_{n_0}
	\leq
	\alpha^{n_0}A_0
	+n_0\alpha^{n+n_0}(\delta_0+\delta_1).
	$
	Thus,
	\begin{align}
		\label{A6}
		d(x_{m+n_0},x_{n+n_0})
		&\leq
		\alpha^{n_0}d(x_m,x_n)
		+n_0\alpha^{n+n_0}(\delta_0+\delta_1).
	\end{align}
	
	Next, applying the $b$-metric inequality twice, we obtain
	\begin{align*}
		d(x_m,x_n)
		\leq
		s\left(d(x_m,x_{m+n_0})
		+d(x_{m+n_0},x_n)\right)
		\leq
		s\left(d(x_m,x_{m+n_0})
		+s[d(x_{m+n_0},x_{n+n_0})
		+d(x_{n+n_0},x_n)]\right).
	\end{align*}
	Hence,
	\begin{align}
		\label{A7}
		d(x_m,x_n)&\leq s\,d(x_m,x_{m+n_0}) +s^2d(x_{m+n_0},x_{n+n_0})
		+s^2d(x_{n+n_0},x_n).
	\end{align}
	We next estimate the two end terms in \eqref{A7}. Define
	$
	P_n=d(x_{n+n_0},x_n),~ n\in\mathbb{N}_0.
	$
	For $n\geq1$, applying \eqref{Equation Paired-contraction} to the
	three points
	$x_{n+n_0-1},x_{n-1},x_n$ gives
	$$
	d(x_{n+n_0},x_n)+d(x_n,x_{n+1})
	\leq
	\alpha\left(
	d(x_{n+n_0-1},x_{n-1})
	+d(x_{n-1},x_n)
	\right).
	$$
	Therefore,
	$
	P_n+\delta_n
	\leq
	\alpha(P_{n-1}+\delta_{n-1}),
	$
	and hence
	$
	P_n
	\leq
	\alpha P_{n-1}+\alpha\delta_{n-1}.
	$
	Iterating this inequality, we obtain
	$$
	P_n
	\leq
	\alpha^nP_0+
	\sum_{k=1}^{n}\alpha^{n-k+1}\delta_{k-1}.
	$$
	By \eqref{A3}, we have
	\begin{equation}
		\label{A8}
		d(x_{n+n_0},x_n)
		\leq
		\alpha^n d(x_{n_0},x_0)
		+n\alpha^n(\delta_0+\delta_1).
	\end{equation}
	Similarly,
	\begin{equation}
		\label{A9}
		d(x_{m+n_0},x_m)
		\leq
		\alpha^m d(x_{n_0},x_0)
		+m\alpha^m(\delta_0+\delta_1).
	\end{equation}
	Substituting \eqref{A6}, \eqref{A8}, and \eqref{A9} into
	\eqref{A7}, we obtain
	\begin{align*}
		d(x_m,x_n)
		&\leq
		s\alpha^m d(x_{n_0},x_0)
		+s m\alpha^m(\delta_0+\delta_1)
		+s^2\alpha^{n_0}d(x_m,x_n)
		+s^2n_0\alpha^{n+n_0}(\delta_0+\delta_1)\\
		&\qquad
		+s^2\alpha^n d(x_{n_0},x_0)
		+s^2n\alpha^n(\delta_0+\delta_1).
	\end{align*}
	Therefore,
	\begin{equation}
		\label{A10}
		\begin{aligned}
			(1-s^2\alpha^{n_0})d(x_m,x_n)
			&\leq
			s\alpha^m d(x_{n_0},x_0)
			+s^2\alpha^n d(x_{n_0},x_0)
			+s m\alpha^m(\delta_0+\delta_1)
			+s^2n\alpha^n(\delta_0+\delta_1)\\
			&\qquad
			+s^2n_0\alpha^{n+n_0}(\delta_0+\delta_1).
		\end{aligned}
	\end{equation}
	Set
	$
	\displaystyle r_n:=\sup_{k\geq n} k\alpha^k.
	$
	Since $0<\alpha<1$, we have
	$
	k\alpha^k\to0
	~\text{as }k\to\infty,
	$
	and hence
	$
	r_n\to 0
	~\text{as }n\to\infty.
	$
	Since $m>n$ and $0<\alpha<1$, we have
	$
	\alpha^m\leq\alpha^n,~
	m\alpha^m\leq r_n,~
	n\alpha^n\leq r_n.
	$
	It follows from \eqref{A10} that
	\begin{align*}
		d(x_m,x_n)
		\leq
		\frac{(s+s^2)\alpha^n}
		{1-s^2\alpha^{n_0}}
		d(x_{n_0},x_0)
		+
		\frac{(s+s^2)r_n+s^2n_0\alpha^{n+n_0}}
		{1-s^2\alpha^{n_0}}
		(\delta_0+\delta_1).
	\end{align*}
	The right-hand side is independent of $m$. Moreover,
	$\alpha^n\to 0,~
	r_n\to 0,~
	\alpha^{n+n_0}\to0
	\text{ as }n\to\infty.$
	Hence
	$
\displaystyle	\lim_{n\to\infty}\sup_{m>n}d(x_m,x_n)=0,
	$
	which proves that $\{x_n\}$ is a Cauchy sequence.
\end{proof}

\begin{theorem}\label{Theorem Paired-contraction}
	Let $(X,d)$ be a complete $b$-metric space with $|X|\geq 3$, and let $T:X\to X$ be a Paired contraction. Then $T$ has a fixed point in $X$ if and only if $T$ has no periodic points of prime period 2. Moreover, the number of fixed points of $T$ is at most two. 
\end{theorem}
\begin{proof}
	Let $T$ be a mapping with no periodic point of prime period 2. Let $x_0\in X$ such that
$x_1=Tx_0 , ~ x_2=Tx_1,~\ldots, ~ x_n=Tx_{n-1}=T^nx_0 .$
If $x_n$ is a fixed point of $T$ for some $n\in \mathbb{N}_0$, then there is nothing to prove. Suppose that $x_n$ is not a fixed point of the mapping $T$ for every $n\in\mathbb{N}_0$, we have $x_{n}\neq Tx_{n}=x_{n+1}$. Additionally, since $T$ has no periodic points of prime period 2, then  $x_{n+1} = T(T(x_{n-1})) \neq x_{n-1}$. Hence $x_{n-1}, x_n, x_{n+1}$ are pairwise distinct. Substituting $x=x_{n-1}, y=x_n, z= x_{n+1}$ in \eqref{Equation Paired-contraction}, we obtain \eqref{Paired-Banach Inequality}. Let 
$p_n = d(x_n, x_{n+1}) + d(x_{n+1}, x_{n+2})$
for all $n\in \mathbb{N}_0$. Then 
\begin{eqnarray}
	\label{estimate1}
	p_1 \leq  \alpha  p_0, ~ p_2 \leq  \alpha  p_1, ~ p_n \leq  \alpha  p_{n-1}\leq \alpha^np_0 \text{ and }
	p_0 > p_1 > \dots > p_n > \dots.
\end{eqnarray}

Next, we show that all points in the orbit $\{x_n:n\in\mathbb{N}_0\}$ are
distinct. Suppose, to the contrary, that $x_i=x_j$ for some
$0\leq i<j$. If $j=i+1$, then $x_i=Tx_i$, contradicting the
assumption that the orbit contains no fixed point. Hence $j\geq i+2$,
and consequently
$x_{j+1}=x_{i+1},~x_{j+2}=x_{i+2}.$
Thus,
\begin{align*}
	p_i
	=d(x_i,x_{i+1})+d(x_{i+1},x_{i+2})
	=d(x_j,x_{j+1})+d(x_{j+1},x_{j+2})
	=p_j,
\end{align*}
which contradicts $p_j<p_i$ for $i<j$ in \eqref{estimate1}.

We are now ready to show that $\{x_n\}$ is a Cauchy sequence. 
If $\alpha=0$, then by \eqref{Equation Paired-contraction}, we have  $d(Tx,Ty)+d(Ty,Tz)=0.$
Consequently, 
$d(Tx,Ty)=d(Ty,Tz)=0$ and hence $Tx=Ty=Tz$
for every pairwise distinct $x,y,z\in X$. Since $|X|\geq 3$, this implies that $T$ is a constant mapping. Thus, there exists $p\in X$ such that
$T(x)=p,~ x\in X.$ In particular, $T(p)=p,$ so $p$ is a fixed point of $T$. Therefore, $T$ has a fixed point. Now suppose $\alpha\in (0,1)$. By Lemma  \ref{Cauchy-Paired-contraction}, $\{x_n\}$ is a Cauchy sequence.

	Let us prove that $Tx^*=x^*$. Consider the set $\mathcal{B}=\{n\in \mathbb{N}_0:x_n=x^*\}$. 
	\paragraph*{\textbf{Case 1}} If $\mathcal{B}$ is an infinite set, then there exists a subsequence $x_{n_k} \to x^*$ and hence $x_{n_k+1}=Tx_{n_k}=Tx^*.$ Since $x_{n_k+1}$ is also a subsequence of a convergent sequence. Thus $x_{n_k+1}\to x^*$. Hence $Tx^*=x^*.$
	 	\paragraph*{\textbf{Case 2}} If $\mathcal{B}$ is a finite set, then there exists $N\in \mathbb{N}$ such that $x_n\neq x^*$ for all $n\geq N$. Since the orbit points are distinct, the points $x^*,x_{n}$ and  $x_{n+1}$ are pairwise distinct for every $n\geq N+1$. Thus, by $b$-metric inequality and inequality \eqref{Equation Paired-contraction} for each $n\geq N$, we have
	\begin{align*}
	d(x^*, Tx^*) &\leq s(d(x^*, x_{n+1}) + d(x_{n+1}, Tx^*))= s(d(x^*, x_{n+1}) + d(Tx_{n}, Tx^*)) \\
	&\leq s d(x^*, x_{n+1}) + s(d(Tx_{n}, Tx^*) + d(Tx^*, Tx_n)) \\
	&\leq s d(x^*, x_{n+1}) + s\alpha (d(x_{n}, x^*) + d(x^*, x_{n+1})).
\end{align*}
Taking the limit as $n \to \infty$, we obtain
	$d(x^*, Tx^*)=0.$
Hence, $Tx^*= x^*$ as desired.  In both cases, we can conclude that $T$ has a fixed point in $X$. 

On the other hand, suppose that $T$ has a fixed point $x\in X$ and a periodic point $y\in X$ of prime period $2$, so that $T(Ty)=y$. Let $z=Ty$. Then $Tz=y$. If $x=y$ or $y=z$, then $Ty=y$, contradicting the fact that $y$ has prime period $2$. Similarly, if $x=z$, then
$Ty=z=x=Tx=Tz=y,$
which again implies $Ty=y$, a contradiction. Hence, $x,y,z$ are pairwise distinct. Applying \eqref{Equation Paired-contraction} to the ordered triple $(x,y,z)$, we obtain
\begin{align}
	\label{A}
	d(x,z)+d(y,z)\leq\alpha\left(d(x,y)+d(y,z)\right).
\end{align}
Similarly, applying \eqref{Equation Paired-contraction} to the
ordered triple $(x,z,y)$ yields
\begin{align}
	\label{B}
	d(x,y)+d(y,z)
\leq
\alpha\left(d(x,z)+d(y,z)\right).
\end{align}
Set
$
A=d(x,z)+d(y,z)>0
~\text{and}~
B=d(x,y)+d(y,z)>0.
$
Then \eqref{A} and \eqref{B} imply
$A\leq\alpha B \text{ and } B\leq\alpha A.$
Therefore,
$A\leq\alpha B\leq\alpha^2 A.$
Since $A>0$, we obtain
$\alpha\geq1.$
This contradicts the assumption
$\alpha\in[0,1).$
Hence, $T$ has no periodic points of prime period $2$. 

Finally, let $x,y$ and $z$ be three distinct fixed points, say $Tx = x$, $Ty = y$, and $Tz = z$. Then, the inequality \eqref{Equation Paired-contraction} yields $\alpha \ge 1$, contradicting the assumption that $\alpha\in [0,1)$. Therefore, $T$ has at most two fixed points. 
\end{proof}

\begin{remark}
Theorem \ref{Theorem Paired-contraction} extends \cite[Theorem 2]{Bouker2026} by enlarging the admissible range of the contraction coefficient from $\alpha\in\left[0,\frac{1}{s}\right)$ to $\alpha\in[0,1)$. Moreover, we establish the converse of \cite[Theorem 2]{Bouker2026}.
Furthermore, by setting $s=1$ in Theorem \ref{Theorem Paired-contraction}, we get \cite[Theorem 2.1]{Chand2024}.
\end{remark}

The following example illustrates that the class of Paired contractions covered by Theorem \ref{Theorem Paired-contraction}  is strictly larger than the class considered in  \cite[Theorem 2]{Bouker2026}.
\begin{example}
	Let $X = \{a, b, c\}$ and define $d: X \times X \to \mathbb{R}_+$ by
	$$
	d(a,b) = d(b,c) = 1, ~ d(a,c) = 4, \text{ and } d(x,x)=0 \text{ for all } x\in X.
	$$
	Clearly, $(X,d)$ is a complete $b$-metric space. Now, define $T: X \to X$ by
	$$
	Ta = b,~Tb= c,~Tc = c.
	$$
	Clearly, $T$ has no periodic points of prime period 2 and $\text{Fix}(T)=\{c\}$. By \eqref{Equation Paired-contraction}, the six permutations of $(a,b,c)$ yield the following lower bounds for $\alpha$: $(a,b,c)$ and $(c,b,a)$ give $\alpha \geq 1/2$; $(a,c,b)$ and $(b,c,a)$ give $\alpha \geq 1/5$; and $(b,a,c)$ and $(c,a,b)$ give $\alpha \geq 2/5$. Hence, $T$ satisfy all the hypotheses of Theorem \ref{Theorem Paired-contraction}  for every
$\alpha\in\left[\frac{1}{2},1\right)$.  On the other hand, since $s=2$, the contraction coefficient in \cite[Theorem 2]{Bouker2026} is required to satisfy
$\alpha\in\left[0,\frac{1}{2}\right).$
\end{example}

The following example shows that the assumption in Theorem \ref{Theorem Paired-contraction} that $T$ has no periodic point of prime period $2$ is essential. Indeed, we construct a Paired contraction $T$ on a $b$-metric space $X$ that has periodic points of prime period $2$ but no fixed point.

\begin{example}
	Let
	$
	X=\{a,b,c\},
	$
	and define $d:X\times X\to[0,\infty)$ by
	$$
	d(a,b)=1,~
	d(a,c)=2, ~
	d(b,c)=6, \text{ and } d(x,x)=0~ \forall ~x\in X.
	$$
	Then $(X,d)$ is a $b$-metric space with coefficient $s=2$.  Define $T:X\to X$ by
	$$
	Ta=b,~ Tb=a,~ Tc=a.
	$$
	Clearly, $T^2a=a,~ Ta\neq a,$
	and
	$T^2b=b,~ Tb\neq b.$
	Thus, $a$ and $b$ are periodic points of $T$ of prime period $2$ and 
	$\operatorname{Fix}(T)=\emptyset.$ For the six permutations $(a,b,c)$, namely $(a,b,c), (a,c,b), (b,a,c),
	(b,c,a),\\ (c,a,b),(c,b,a),$
	the corresponding ratios
	are, respectively,
	$
	\frac{1}{7}, \frac{1}{8}, \frac{2}{3},
	\frac{1}{8}, \frac{2}{3}, \frac{1}{7}.$ 
	Hence, the maximum is $\frac{2}{3}$, and therefore $T$ satisfies the Paired contraction condition for
$\alpha\in\left[\frac{2}{3},1\right).$
\end{example}

The next example demonstrates the applicability of Theorem \ref{Theorem Paired-contraction} and illustrates the significance of its hypotheses.
           \begin{example}
           	Let $X = \{0\} \cup [1, 2] \cup \{3\}$, and let $d : X \times X \to [0, \infty)$ be defined by
           	$$
           	d(x, y) = |x - y|^2~\text{for all } x, y \in X.$$
           	Then $(X, d)$ is a complete $b$-metric space with $s = 2$. Define a mapping $T : X \to X$ by
           \begin{equation*}
           		T x =
           	\begin{cases}
           		\frac{x}{2} + 1 & \text{if } x \in [1, 2], \\[0.5mm]
           		1 & \text{if } x = 0, \\[0.5mm]
           		2 & \text{if } x = 3.
           	\end{cases}
           \end{equation*}
           	We claim that $T$ is a Paired contraction with $\alpha = \frac{1}{4}$. We verify this in three cases.           	
           	\subsection*{Case 1: All $x, y, z \in [1, 2]$.}
           	Then
           	$$
           	d(Tx, Ty) + d(Ty, Tz)
           	= \frac{1}{4} \left( (x-y)^2 + (y-z)^2 \right)
           	= \frac{1}{4} \left( d(x,y) + d(y,z) \right).
           	$$
           	\subsection*{Case 2: Two points in $[1,2]$ and one isolated point.}
           	Let $x, y \in [1,2]$, $x \neq y$, and let $z$ be an isolated point.
           	
           	\paragraph{\textbf{Subcase 2.1: $z = 0$.}}
           	Then $T x = \frac{x}{2} + 1$, $T y = \frac{y}{2} + 1$, and $T(0) = 1$. 
           	Therefore
           	$$
           	d(Tx, Ty) + d(Ty, T0)
           	= \frac{1}{4}\left((x-y)^2 + y^2\right)
           	= \frac{1}{4} \left( d(x, y) + d(y, 0) \right).
           	$$
           	
           	\paragraph{\textbf{Subcase 2.2: $z = 3$.}}
           	Then $T x = \frac{x}{2} + 1$, $T y = \frac{y}{2} + 1$, and $T(3) = 2$.
           	Therefore
           	$$
           	d(Tx, Ty) + d(Ty, T3)=\frac{1}{4}
           	\left((x-y)^2 + (2-y)^2\right)	\leq \frac{1}{4} \left( d(x, y) + d(y, 3) \right),$$
           	since $(2-y)^2 \leq (3-y)^2$ for $y \in [1,2]$. 
           	
           	\paragraph{\textbf{Subcase 2.3: $y = 0$.}}
           	Let $x,z\in[1,2]$ with $x\neq z$. Then
           	$
           	Tx=\frac{x}{2}+1,~ T0=1,~ Tz=\frac{z}{2}+1.
           	$
           	Hence,
           	\begin{align*}
           		d(Tx,T0)+d(T0,Tz)=\frac{x^2}{4}+\frac{z^2}{4}=\frac14\left(d(x,0)+d(0,z)\right).
           	\end{align*}       	
           	\paragraph{\textbf{Subcase 2.4: $y = 3$.}}
           	Let $x,z\in[1,2]$ with $x\neq z$, and let $y=3$. Then
           	$$
           	Tx=\frac{x}{2}+1,~ T3=2,~ Tz=\frac{z}{2}+1.
           	$$
           	Therefore,
           	\begin{align*}
           		d(Tx,T3)+d(T3,Tz)
           		=\frac{(2-x)^2}{4}+\frac{(2-z)^2}{4}\leq \frac14\left(d(x,3)+d(3,z)\right),
           	\end{align*}
           	since
           	$
           	(2-x)^2\leq(3-x)^2
           	~\text{and}~
           	(2-z)^2\leq(3-z)^2
           	$  for $x,z\in[1,2]$.
           
           	\subsection*{Case 3: One point in $[1,2]$ and two isolated points.}
           	
           	Since $x,y,z$ are pairwise distinct, the two isolated points must be $0$ and $3$, while the remaining point belongs to $[1,2]$. Up to the interchange of $x$ and $z$, there are two possibilities.       	
           	\paragraph{\textbf{Subcase 3.1: $x\in[1,2]$, $y=3$, and $z=0$.}}
           	Then
           	$
           	Tx=\frac{x}{2}+1,~ T3=2,~ T0=1.
           	$
           	Hence,
           	\begin{align*}
           		d(Tx,T3)+d(T3,T0)=\left|\frac{x}{2}+1-2\right|^2+|2-1|^2=\frac{(2-x)^2}{4}+1\leq \frac14\left(d(x,3)+d(3,0)\right),
           	\end{align*}
           	Since $1\leq \frac{9}{4}$ and $(2-x)^2\leq(3-x)^2$ for $x\in[1,2]$.
           	The case $x=0$, $y=3$, $z\in[1,2]$ follows by interchanging $x$ and $z$.
           	
           	\paragraph{\textbf{Subcase 3.2: $x\in[1,2]$, $y=0$, and $z=3$.}}
            Then $T x = \frac{x}{2} + 1$, $T(0) = 1$, $T(3) = 2$.
           	Thus
           $$
           	d(Tx, T0) + d(T0, T3)
           	= \frac{x^2}{4} + 1	\leq \frac{1}{4} \left( d(x, 0) + d(0, 3) \right).
           	$$

           	In all cases, we have shown that
           	\[
           	d(Tx, Ty) + d(Ty, Tz)
           	\leq \frac{1}{4} \left( d(x, y) + d(y, z) \right)
           	\]
           	for all pairwise distinct $x, y, z \in X$. Therefore, $T$ is a Paired contraction with $\alpha = \frac{1}{4}$. Moreover, $T$ has no periodic points of prime period $2$. Indeed, for $x \in [1,2]$,
           	$$
           	T^2 x = T\left(\frac{x}{2} + 1\right) = \frac{1}{2}\left(\frac{x}{2} + 1\right) + 1 = \frac{x}{4} + \frac{3}{2} \neq x,$$
           	unless $x = 2$, but $T(2) = 2$ is a fixed point.
           		Thus, by Theorem \ref{Theorem Paired-contraction}, $\text{Fix}(T)=\{2\}.$
           \end{example}

As a consequence of Theorem \ref{Theorem Paired-contraction}, we recover \cite[Theorem 2.1]{Mitrovic} as a corollary.
\begin{corollary}
	\label{Corollary-Banach}
	Let $(X, d)$ be a nonempty complete $b$-metric space with a mapping $T:X \to X$ such that for all $x,y\in X$
	\begin{eqnarray}
		\label{CG-Banach}
		d(Tx,Ty)\leq \alpha d(x,y)
	\end{eqnarray}
	with $\alpha\in [0,1)$. Then $T$ has a
	unique fixed point.
\end{corollary}

\begin{proof}
If $|X|=1$, the unique point of $X$ is trivially a fixed point.	If $|X| = 2$, then there are four possible cases as follows: $Tx=y$ and $Ty=y$, $Tx=x$ and $Ty=x$, $Tx=y$ and $Ty=x$, $Tx=x$ and $Ty=y$. The first two cases yield a unique fixed point, while the remaining two cases would force $\alpha\geq 1$, thereby contradicting the assumption that $\alpha\in [0,1)$.
	
	Now, suppose that $|X| \ge 3$. Assume, for contradiction, that there exists $x \in X$ such that $T(Tx) = x$. Then $d(Tx, x) = d(Tx, T(Tx))\leq \alpha d(Tx,x)$, which contradicts that $\alpha\in [0,1)$. Hence, $T$ has no periodic points of prime period 2. Let $x,y,z$ be distinct points in $X$. Applying \eqref{CG-Banach} to the pairs $\{y,z\}$ and summing the resulting inequalities yields \eqref{Equation Paired-contraction}. Thus, by Theorem \ref{Theorem Paired-contraction},  $T$ must have a fixed point. Finally, suppose $Tx = x$ and $Ty = y$ with $x \neq y$. Then
	$$d(x, y) = d(Tx, Ty) \leq \alpha \, d(x, y),$$
	which implies $\alpha \geq 1$, a contradiction. Thus $x = y$. Therefore $T$ has a unique fixed point.
\end{proof}

\section{paired-Kannan Contractions}
In this section, we investigate Paired--Kannan contractions in $b$-metric spaces. We observe that the corresponding result in \cite[Theorem 3]{Bouker2026} assumes that the contraction constant belongs to $\left[0,\frac{1}{2}\right)$. However, when $s>2$, this condition is insufficient for the $b$-metric limit argument used to establish the fixed point property. To address this issue, we explicitly incorporate the missing dependence on the $b$-metric coefficient into Definition \ref{Withouts} as follows:

\begin{definition} 
	Let $(X,d)$ be a $b$-metric space with $|X| \geq 3$. A mapping $T: X \to X$ is called \textit{paired-Kannan contraction} if there exists a constant 
	\begin{equation}
		\label{alpha}
		\alpha<\begin{cases}
			\frac{1}{2}, ~~1\leq s\leq 2;\\
			\frac{1}{s}, ~~s>2.
		\end{cases}
	\end{equation}
	such that the inequality \eqref{Equation paired-Kannan}
	holds for all pairwise distinct $x, y, z \in X$.
\end{definition}
\begin{remark}
	The condition \eqref{alpha} is equivalent to $\alpha<\min\left\{\frac{1}{2},\frac{1}{s}\right\}.$
Moreso, the Kannan-type contraction and the paired-Kannan contraction are distinct (see  \cite[Example 2.6]{Chand2024b}).
\end{remark}
\begin{proposition}
	The hypotheses of \cite[Theorem 3]{Bouker2026} do not imply the existence of a fixed point when $s>2$.
\end{proposition}
We present an example showing that the stated hypotheses of \cite[Theorem 3]{Bouker2026} do not guarantee the asserted conclusion.  By constructing a complete $b$-metric space $(X, d)$ with coefficient $s = 5$ and a mapping $T: X \to X$ satisfying all the hypotheses of  \cite[Theorem 3]{Bouker2026} with $\alpha = \frac{2}{5} < \frac{1}{2}$, but having no fixed points. 
\begin{example}
	\label{counterexample}
Let $X = \mathbb{N}_0 \cup \{p, q\},$
where $\mathbb{N}_0 = \{0, 1, 2, \dots\}$. Define $d : X \times X \to [0, \infty)$ by 
\begin{eqnarray*}
	\begin{cases}
		d(x, x) = 0,\\
		d(m, n) = 10^{-m} + 10^{-n},~m, n \in \mathbb{N}_0, \\
		d(q, n) = 10^{-n},~n \in \mathbb{N}_0, \\
		d(p, n) = \frac{1}{5},~n \in \mathbb{N}_0,\\
		d(p, q) = 1. 
	\end{cases}
\end{eqnarray*}
\subsection*{\textit{Verify that $(X,d)$ is a $b$-metric space}}
 Symmetry is clear from the definition. We verify the $b$-metric inequality as follows:
\subsubsection*{Case 1: $x, z \in \mathbb{N}_0$, $y = p$}
$d(x, z) \leq 2 = 5 \left( \frac{1}{5} + \frac{1}{5} \right) =5 ( d(x, p) + d(p, z)).$
\subsubsection*{Case 2: $x, z \in \mathbb{N}_0$, $y = q$}
$
d(x, z) = 10^{-x} + 10^{-z} \leq 5 ( 10^{-x} + 10^{-z} ) = 5 ( d(x, q) + d(q, z)).$

\subsubsection*{Case 3: $x = p$, $z = q$, $y \in \mathbb{N}_0$}
$d(p, q) = 1 \leq 5 \left( \frac{1}{5} + 10^{-y} \right)=5 ( d(p, y) + d(y, q)).$

\subsubsection*{Case 4: $x = p$, $z \in \mathbb{N}_0$, $y = q$}
$ d(p, z) = \frac{1}{5}\leq 5 \left( 1 + 10^{-z} \right)=5 ( d(p, q) + d(q, z)).$

\subsubsection*{Case 5: $x = q$, $z \in \mathbb{N}_0$, $y = p$}
$d(q, z) = 10^{-z} \leq 1\leq 5 \left( 1 + \frac{1}{5} \right)=5 ( d(q, p) + d(p, z)).$\\
	The remaining cases are either trivial or follow by symmetry. Hence,
$(X,d)$ is a $b$-metric space with coefficient $s=5$.

\subsection*{\textit{Show that the sequence is Cauchy}}
Let $\{x_n\}$ be the sequence defined by $x_n=n$. For any
$m,n\in\mathbb{N}_0$, we have
$d(m,n)=10^{-m}+10^{-n}.$
Let $\varepsilon>0$. Choose $N\in\mathbb{N}$ such that
$10^{-N}<\frac{\varepsilon}{2}.$
Then, for all $m,n\geq N$,
$$
d(m,n)
=10^{-m}+10^{-n}
\leq 10^{-N}+10^{-N}
=2\cdot10^{-N}
<\varepsilon.
$$
Therefore, $\{x_n\}_{n\in\mathbb{N}_0}$ is a Cauchy sequence in $(X,d)$. In fact, the sequence converges to $q$, since
$
d(x_n,q)=d(n,q)=10^{-n}\to0.$
Hence, $x_n\to q.$

\subsection*{\textit{We now show that $(X,d)$ is complete}}
Let $(x_j)$ be a Cauchy sequence in $X$. If the range of $(x_j)$ is finite, then
$$
\delta:=\min\{d(u,v):u,v\in\operatorname{Ran}\{x_j\},\ u\neq v\}>0.
$$
Since $(x_j)$ is Cauchy, there exists $N$ such that
$
d(x_j,x_k)<\delta,~ j,k\geq N.
$
Hence $x_j=x_k$ for all $j,k\geq N$, so the sequence is eventually constant and therefore convergent. Suppose now that the range of $(x_j)$ is infinite. Choose $N$ such that
$
d(x_j,x_k)<\frac15,~ j,k\geq N.
$
Since $d(p,n)=1/5$ for every $n\in\mathbb N_0$, the tail $\{x_j:j\geq N\}$ cannot contain both $p$ and a natural number. If $p$ occurred infinitely often, the Cauchy property would therefore imply that only finitely many distinct natural numbers could occur in the tail, contradicting the assumption that the range is infinite. Thus, $p$ occurs only finitely many times. Consequently, for all sufficiently large $j$,
$
x_j\in\mathbb N_0\cup\{q\}.
$
Since the range is infinite, infinitely many distinct natural
numbers occur. We claim that the values of these natural-number
terms tend to infinity. Indeed, if some $m\in\mathbb N_0$ occurred
infinitely often while arbitrarily large natural numbers also
occurred, then for a sequence $(n_k)$ of such natural numbers with
$n_k\to\infty$,
$$d(m,n_k)=10^{-m}+10^{-n_k}\to10^{-m}>0,$$
contradicting the Cauchy property. Hence, the natural-number
values occurring in the sequence tend to infinity. Since
$d(n,q)=10^{-n}$, it follows that
$d(x_j,q)\to0.$ Thus $x_j\to q$.
Therefore every Cauchy sequence in $X$ converges in $X$, and hence $(X,d)$ is complete.

\subsection*{ \textit{Define the Mapping}}
Define $T : X \to X$ by
$Tq = p,~Tp = 2,~Tn = n + 1~(n \in \mathbb{N}_0).
$
Clearly, 
$Tq = p \neq q, ~Tp = 2 \neq p,~ Tn = n + 1 \neq n,$ so $T$ has no fixed point. 
Moreover, 
$$T^2 q = Tp = 2 \neq q,~T^2 p = T2 = 3 \neq p,~T^2 n = n + 2 \neq n.$$ Hence, $T$ has no periodic points of prime period 2. 

	\subsection*{\textit{Verify the paired-Kannan Hypotheses}}
	Since both sides of \eqref{Equation paired-Kannan} are unchanged when $x$ and $z$ are
	interchanged, it is enough to consider one representative from each
	pair obtained by the interchange $x\leftrightarrow z$.
	
	\subsubsection*{Case A: $x,y,z\in\mathbb N_0$.}
 Let $x=m$, $y=n$, and $z=k$, where $m,n,k\in\mathbb{N}_0$ are pairwise distinct. Since
$
Tm=m+1,~ Tn=n+1,~ Tk=k+1,
$
we obtain
\begin{align*}
	d(Tm,Tn)+d(Tn,Tk)=d(m+1,n+1)+d(n+1,k+1)
	\leq 0.2\left(10^{-m}+10^{-n}+10^{-k}\right).
\end{align*}
Moreover,
\begin{align*}
	d(m,Tm)+d(n,Tn)+d(k,Tk)=d(m,m+1)+d(n,n+1)+d(k,k+1)
	=1.1	\left(10^{-m}+10^{-n}+10^{-k}\right).
\end{align*}
Thus
\begin{align*}
	d(Tm,Tn)+d(Tn,Tk)\leq
	\frac{2}{11}
	\left(10^{-m}+10^{-n}+10^{-k}\right)=\frac{2}{11}
	\left[d(m,Tm)+d(n,Tn)+d(k,Tk)\right].
\end{align*}
Therefore, \eqref{Equation paired-Kannan} holds for coefficient $\frac{2}{11}$.
	
	\subsubsection*{Case B: Two points belong to $\mathbb N_0$ and the third is $p$.}
	
	Let $m,n\in\mathbb N_0$, $m\neq n$.
	
	\paragraph{\textbf{Subcase B.1: $y=p$.}}
	Let $x=m$, $y=p$, and $z=n$. Then
	$
	Tm=m+1,~ Tp=2,~ Tn=n+1.
	$
	Consequently,
	$$
	\begin{aligned}
		d(Tm,Tp)+d(Tp,Tn)
		=d(m+1,2)+d(2,n+1)\leq
		\frac{1}{10}\left(10^{-m}+10^{-n}\right)
		+\frac{1}{50}.
	\end{aligned}
	$$
	Moreover,
	$$
	d(m,Tm)+d(p,Tp)+d(n,Tn)= \frac{11}{10}\left(10^{-m}+10^{-n}\right)+\frac15.
	$$
	Hence,
	$$
	d(Tm,Tp)+d(Tp,Tn)
	\leq
	\frac{1}{10}
	\left(d(m,Tm)+d(p,Tp)+d(n,Tn)\right).
	$$
	Therefore, \eqref{Equation paired-Kannan} holds for coefficient $\frac{1}{10}$.
	\paragraph{\textbf{Subcase B.2: $p$ is an endpoint.}}
	
By symmetry, it is enough to consider $x=m$, $y=n$, and $z=p$. Then
\begin{align*}
	d(Tm,Tn)+d(Tn,Tp)
	=d(m+1,n+1)+d(n+1,2)\leq
	\frac{1}{10}10^{-m}+\frac{1}{5}10^{-n} +\frac{1}{100}.
\end{align*}
On the other hand,
$$
d(m,Tm)+d(n,Tn)+d(p,Tp)
=
\frac{11}{10}\left(10^{-m}+10^{-n}\right)+\frac15.
$$
Therefore,
$$
d(Tm,Tn)+d(Tn,Tp)
\leq
\frac{21}{130}
\left(d(m,Tm)+d(n,Tn)+d(p,Tp)\right).
$$	
	\subsubsection*{Case C: Two points belong to $\mathbb N_0$ and the third is $q$.}
	
	Let $m,n\in\mathbb N_0$, $m\neq n$.
	
	\paragraph{\textbf{Subcase C.1: $y=q$.}}
	
Let $x=m$, $y=q$, and $z=n$. Then
$
Tm=m+1,~ Tq=p,~ Tn=n+1.
$
Consequently,
\begin{align*}
	d(Tm,Tq)+d(Tq,Tn)
	=d(m+1,p)+d(p,n+1)
	=\frac25.
\end{align*}
Moreover,
$$
\begin{aligned}
	d(m,Tm)+d(q,Tq)+d(n,Tn)
	=d(m,m+1)+d(q,p)+d(n,n+1)
	=1.1\left(10^{-m}+10^{-n}\right)+1.
\end{aligned}
$$
Therefore,
$$
d(Tm,Tq)+d(Tq,Tn)
\leq
\frac25\left(d(m,Tm)+d(q,Tq)+d(n,Tn)\right).
$$

	\paragraph{\textbf{Subcase C.2: $q$ is an endpoint.}}
	
By symmetry, let $x=m$, $y=n$, and $z=q$. Then
$$
d(Tm,Tn)+d(Tn,Tq)
=d(m+1,n+1)+d(n+1,p).
$$
Consequently,
$$
\begin{aligned}
	d(Tm,Tn)+d(Tn,Tq)=
	0.1\left(10^{-m}+10^{-n}\right)+\frac15.
\end{aligned}
$$
On the other hand,
$
	d(m,Tm)+d(n,Tn)+d(q,Tq)
	=1.1\left(10^{-m}+10^{-n}\right)+1.
$
Therefore,
$$
d(Tm,Tn)+d(Tn,Tq)
\leq
\frac15
\left(d(m,Tm)+d(n,Tn)+d(q,Tq)\right).
$$
Hence \eqref{Equation paired-Kannan}  holds in this case with the coefficient $\frac15$.
	
	\subsubsection*{Case D: One point belongs to $\mathbb N_0$ and the other two are $p$ and $q$.}
	
	Let $n\in\mathbb N_0$. Up to the interchange of $x$ and $z$, there
	are three possibilities according to the position of the natural
	number and the isolated points.
	
\paragraph{\textbf{Subcase D.1: $x=p$, $y=q$, $z=n$.}}

Then
$
Tp=2,~ Tq=p,~ Tn=n+1,
$
and hence
\begin{align*}
	d(Tp,Tq)+d(Tq,Tn)=d(2,p)+d(p,n+1)=\frac25.
\end{align*}
Moreover,
\begin{align*}
	d(p,Tp)+d(q,Tq)+d(n,Tn)
	=d(p,2)+d(q,p)+d(n,n+1)
	=\frac65+\frac{11}{10}10^{-n}.
\end{align*}
Thus,
$$
d(Tp,Tq)+d(Tq,Tn)
\leq
\frac13
\left(d(p,Tp)+d(q,Tq)+d(n,Tn)\right).
$$
\paragraph{\textbf{Subcase D.2: $x=q$, $y=p$, $z=n$.}}

We have
$
Tq=p,~ Tp=2,~ Tn=n+1.
$
Consequently,
$$
\begin{aligned}
	d(Tq,Tp)+d(Tp,Tn)=d(p,2)+d(2,n+1)\leq
	\frac{21}{100}+\frac1{10}10^{-n}.
\end{aligned}
$$
On the other hand,
$$
\begin{aligned}
	d(q,Tq)+d(p,Tp)+d(n,Tn)
	=d(q,p)+d(p,2)+d(n,n+1)
	=\frac65+\frac{11}{10}10^{-n}.
\end{aligned}
$$
Thus,
$$
d(Tq,Tp)+d(Tp,Tn)
\leq
\frac{31}{230}
\left(d(q,Tq)+d(p,Tp)+d(n,Tn)\right).
$$
\paragraph{\textbf{Subcase D.3: $x=p$, $y=n$, $z=q$.}}
Then
$
Tp=2,~ Tn=n+1,~ Tq=p.
$
Consequently,
\begin{align*}
	d(Tp,Tn)+d(Tn,Tq)
	&=d(2,n+1)+d(n+1,p)
	\leq
	\frac{21}{100}+\frac1{10}10^{-n},
\end{align*}
and 
	\begin{align*}
	d(p,Tp)+d(n,Tn)+d(q,Tq)&=\frac65+\frac{11}{10}10^{-n}.
\end{align*}
Therefore,
$$
d(Tp,Tn)+d(Tn,Tq)
\leq
\frac{31}{230}
\left(d(p,Tp)+d(n,Tn)+d(q,Tq)\right).
$$	
	The remaining configurations in Cases B--D are obtained by
	interchanging $x$ and $z$, and hence are already covered. Therefore  \eqref{Equation paired-Kannan} holds for every pairwise distinct	$x,y,z\in X$. Thus, $T$ is a Paired-Kannan contraction  with	$
	\alpha=\max\{\frac{31}{230},\frac{1}{10},\frac{2}{5},\frac{1}{5},\frac{2}{11},\frac{1}{3}\}=\frac25<\frac12.$
	Nevertheless, $T$ has no periodic point of prime
	period $2$. Hence, the hypotheses of \cite[Theorem 3]{Bouker2026} are
	satisfied, but $\text{Fix}(T)=\emptyset$. 	In particular, Example \ref{counterexample} shows that the hypothesis $\alpha < 1/2$ is insufficient for Paired-Kannan contraction in $b$-metric spaces when $s > 2$. Theorem \ref{Theorem paired-Kannan} therefore provides a corresponding fixed point result with sufficient condition. 

\end{example}

\begin{theorem}
		\label{Theorem paired-Kannan}
	Let $(X,d)$ be a complete $b$-metric space with $|X|\geq 3$, and let
	$T:X\to X$ be a Paired-Kannan contraction.  Then $T$ has a fixed point in $X$ if and only if $T$ has no periodic points of prime period 2. Moreover, the number of fixed points of $T$ is at most two. 
\end{theorem}
\begin{proof}
	Let $T$ be a mapping with no periodic point of prime period 2. Let $x_0\in X$ such that
$x_1=Tx_0 , ~ x_2=Tx_1,~\ldots, ~ x_n=Tx_{n-1}=T^nx_0 .$
If $x_n$ is a fixed point of $T$ for some $n\in \mathbb{N}_0$, then there is nothing to prove. Suppose that $x_n$ is not a fixed point of the mapping $T$ for every $n\in\mathbb{N}_0$, we have $x_{n}\neq Tx_{n}=x_{n+1}$. Additionally, since $T$ has no periodic points of prime period 2, then  $x_{n+1} = T(T(x_{n-1})) \neq x_{n-1}$. Hence $x_{n-1}, x_n, x_{n+1}$ are pairwise distinct. Putting $x=x_{n-1}, y=x_n, z= x_{n+1}$ in \eqref{Equation paired-Kannan}, we have
	\begin{align*}
		d(Tx_{n-1}, Tx_n) + d(Tx_n, Tx_{n+1}) &\leq \alpha (d(x_{n-1}, Tx_{n-1}) + d(x_n, Tx_n) + d(x_{n+1}, Tx_{n+1})) \\
		d(x_n, x_{n+1}) + d(x_{n+1}, x_{n+2}) &\leq \alpha (d(x_{n-1}, x_n) + d(x_n, x_{n+1}) + d(x_{n+1}, x_{n+2}))
	\end{align*}
	Thus, by simple rearrangement, we obtain
\begin{align*}
	d(x_n, x_{n+1})\leq d(x_n, x_{n+1}) + d(x_{n+1}, x_{n+2}) 
	\leq \frac{\alpha}{1 - \alpha} d(x_{n-1}, x_n)
	\leq \gamma d(x_{n-1}, x_n),
\end{align*}
	where $\displaystyle \gamma=\frac{\alpha}{1 - \alpha}. $
If $\alpha \in [0, \frac{1}{2})$, then $\gamma \in [0,1)$. Also, if $\alpha<\frac{1}{s}$ for $s>2$, then $\gamma \in [0,1)$.  In either case, $\gamma \in [0,1)$.  Thus,
\begin{equation}
	\label{**}
		d(x_n, x_{n+1})\leq \gamma d(x_{n-1}, x_n), ~\gamma  \in [0,1).
\end{equation}
By Lemma \ref{lem:cauchy}, $\{x_n\}$ is a Cauchy sequence. Since $(X,d)$ is complete, $x_n\to x^* \in X$ as $n\to\infty.$  

Now, we want to prove that $Tx^*=x^*$. Consider the set $\mathcal{D}=\{n\in \mathbb{N}:x_n=x^*\}$ in the following two cases:
\paragraph*{\textbf{Case 1}} If $\mathcal{D}$ is an infinite set, then there exists a subsequence $x_{n_k} \to x^*$ and hence $x_{n_k+1}=Tx_{n_k}=Tx^*.$ Since $x_{n_k+1}$ is also a subsequence of a convergent sequence. Thus $x_{n_k+1}\to x^*$. Hence $Tx^*=x^*.$

\paragraph*{\textbf{Case 2}} If $\mathcal{D}$ is a finite set, then there exists $N\in \mathbb{N}$ such that $x_n\neq x^*$ for all $n\geq 0$. We first show that all points of the orbit $\{x_n:n\in\mathbb{N}_0\}$ are distinct. Suppose, to the contrary, that $x_i=x_j$ for some $0\leq i<j$. If $j=i+1$, then
$x_i=Tx_i$, contradicting the assumption that no point of the
orbit is fixed. Hence $j\geq i+2$. Since $x_i=x_j$, we have
$x_{i+1}=Tx_i=Tx_j=x_{j+1}.$ Therefore,
$d(x_i,x_{i+1})=d(x_j,x_{j+1}).$
On the other hand, from the estimate \eqref{**}, 
we obtain
$$
d(x_j,x_{j+1})
\leq
\gamma^{j-i}d(x_i,x_{i+1})
<
d(x_i,x_{i+1}),
$$
which is a contradiction. Hence,
$
x_i\neq x_j~\text{for all }i\neq j.
$
Thus, the orbit $\{x_n:n\in\mathbb{N}_0\}$ consists of pairwise distinct
points. Since the orbit points are distinct, the points $x^*,x_{n}$ and  $x_{n+1}$ are pairwise distinct for every $n\geq N+1$. By the $b$-metric inequality and \eqref{Equation paired-Kannan}, we have
	\begin{align*}
		d(x_{n+1}, Tx^*) &\leq d(Tx_n,Tx^*)+d(Tx^*,Tx_{n+1})\\
		&\leq  \alpha(d(x_n,x_{n+1})+d(x^*,Tx^*)+d(x_{n+1},x_{n+2}))\\
		&\leq \alpha(d(x_n,x_{n+1})+sd(x^*,x_{n+1})+sd(x_{n+1}, Tx^*)+d(x_{n+1},x_{n+2})).
	\end{align*}
	Thus, 
	\begin{align*}
		(1-\alpha s)	d(x_{n+1}, Tx^*)\leq \alpha(d(x_n,x_{n+1})+sd(x^*,x_{n+1})+d(x_{n+1},x_{n+2})).
	\end{align*}
Since $1-\alpha s>0$ by the hypothesis, we have 
\begin{align*}
0\leq d(x_{n+1}, Tx^*)\leq\frac{\alpha}{(1-\alpha s)} (d(x_n,x_{n+1})+sd(x^*,x_{n+1})+d(x_{n+1},x_{n+2})).
\end{align*}
The right-hand side tends to zero as $n\to\infty$. Thus $$\displaystyle \lim_{n \to \infty}	d(x_{n+1}, Tx^*) = 0.$$ 
	 By the $b$-metric inequality, we have 
	$$d(x^*,Tx^*)\leq s(d(x^*,x_{n+1})+d(x_{n+1},Tx^*)).$$
Taking the limit as $n \to \infty$, we obtain $d(x^*, Tx^*) = 0$. Therefore, $x^*$ is a fixed point of $T$.

On the other hand, suppose that $T$ has a fixed point $x\in X$ and a periodic point $y\in X$ of prime period $2$, so that $T(Ty)=y$. Let $z=Ty$. Then $Tz=y$. If $x=y$ or $y=z$, then $Ty=y$, contradicting the fact that $y$ has prime period $2$. Similarly, if $x=z$, then
$Ty=z=x=Tx=Tz=y,$ which again implies $Ty=y$, a contradiction. Hence, $x,y,z$ are pairwise distinct.  In this case, applying inequality \eqref{Equation paired-Kannan}, we obtain
\begin{align*}
		d(x,z) + d(z,y)& \leq \alpha (d(y, z) + d(z, y)).
\end{align*}
Since $d(z,y)\leq d(x,z) + d(z,y)$, we have $\alpha\geq \frac{1}{2}$, contradicting the assumption in \eqref{alpha}. Hence, $T$ has no periodic point of prime period 2. 

Finally, suppose $T$ has three distinct fixed points, say $Tx = x$, $Ty = y$, and $Tz = z$. Then, \eqref{Equation paired-Kannan} yields $d(x,y)+d(y,z)= 0$, which contradicts the distinctness of $x,y$ and $z$. Hence, $T$ has at most two fixed points.
\end{proof}

\begin{remark}
	By setting $s=1$ in Theorem \ref{Theorem paired-Kannan}, we get \cite[Theorem 3.2]{Chand2024b}. Theorem \ref{Theorem paired-Kannan} establishes the corresponding result of \cite[Theorem 3]{Bouker2026} under the complete and appropriate hypotheses. Additionally, we establish the converse of \cite[Theorem 3]{Bouker2026}.
\end{remark}

The following example demonstrates the applicability of Theorem \ref{Theorem paired-Kannan}.

\begin{example}
	Let $ X=\left\{0,\frac13,\frac23,1\right\},$
	and define $d:X\times X\to[0,\infty)$ by
	$$
	d(0,\tfrac13)=5,~
	d(\tfrac13,\tfrac23)=2,~
	d(\tfrac23,1)=5,~
	d(0,\tfrac23)=2,~
	d(\tfrac13,1)=2,~
	d(0,1)=1,
	$$
	with $d(x,x)=0$ and $d(x,y)=d(y,x)$ for all $x,y\in X$.
	Then $(X,d)$ is a complete $b$-metric space with $s=\frac{5}{3}$. Indeed, the only nontrivial $b$-metric inequalities are \\
	$
	d(0,\tfrac23)=2\leq \frac{5}{3}\left(d(0,\tfrac13)+d(\tfrac13,\tfrac23)\right)=\frac{5}{3}(5+2),
	$
	and
	$d(0,\tfrac13)=5\leq \frac{5}{3}\left(d(0,1)+d(1,\tfrac13)\right)=\frac{5}{3}(1+2).
	$
	Define $T:X\to X$ by
	$$
	Tx=
	\begin{cases}
		0, & x\in\left\{0,\frac13\right\},\\[1mm]
		1, & x\in\left\{\frac23,1\right\}.
	\end{cases}
	$$
	Clearly, $\operatorname{Fix}(T)=\left\{0,1\right\}.$
	Moreover, $T$ has no periodic point of prime period $2$, since
	$	T(T\left(\frac13\right))=T(0)=0\neq \frac{1}{3}$ and $
	T(T\left(\frac23\right))=T(1)=1\neq \frac{2}{3}.$
	
	We now verify the inequality \eqref{Equation paired-Kannan}
	for all pairwise distinct $x,y,z\in X$. Since $X$ contains four points, there are 24 ordered triples of pairwise distinct points. 
	However, \eqref{Equation paired-Kannan} is invariant under the interchange of $x$ and $z$, so it suffices to consider the corresponding $12$ triples. Direct computation shows that the inequality holds for $\alpha\geq\frac{1}{10},\frac{1}{7},\frac{2}{7}, \frac{2}{5}$, depending on the triple. Thus, the largest required value is $\alpha=\frac25<\frac35=\frac1s$. Consequently, the inequality \eqref{Equation paired-Kannan} holds for every $ \alpha\in\left[\frac25,\frac{3}{5}\right).$
Therefore, all the hypotheses of Theorem \ref{Theorem paired-Kannan} are satisfied. 
\end{example}

\section{Conclusion}

In this paper, we investigated paired contractive mappings in $b$-metric spaces and established fixed point results for Paired, Paired–Kannan, and Paired–Chatterjea contractions. In particular, we introduced the Paired–Chatterjea contraction framework and established its corresponding fixed point results in $b$-metric spaces.

Our results for paired contraction mappings in b-metric spaces extend the admissible range of the contraction constant from $\left[0,\frac{1}{s}\right)$, as considered in \cite[Theorem 2]{Bouker2026}, to [0,1).  We also examined the coefficient restriction in the existing Paired–Kannan fixed point result of \cite[Theorem 3]{Bouker2026}. The example presented in this paper shows that the stated condition does not guarantee the claimed fixed point conclusion in $b$-metric spaces. This motivates the revised coefficient condition
$\alpha<\min\left\{\frac12,\frac1s\right\}.$ Under this condition, we establish the corresponding fixed point result.

The examples provided throughout the paper illustrate the distinctions among the paired contractive classes and demonstrate the relevance of the stated coefficient restrictions. The results contribute to the study of paired contractive mappings in generalized metric spaces and provide several directions for further investigation.

\end{document}